\documentclass[12pt]{amsart}
\usepackage{amssymb,mathrsfs,color,bbold}

\makeatletter

\newcommand{\Rmnum}[1]{\expandafter\@slowromancap\romannumeral #1@}
\makeatother

\theoremstyle{plain}
\newtheorem{theorem}{Theorem}[section]
\newtheorem{lemma}[theorem]{Lemma}
\newtheorem{proposition}[theorem]{Proposition}
\newtheorem{corollary}[theorem]{Corollary}

\theoremstyle{definition}
\newtheorem{definition}[theorem]{Definition}
\newtheorem{remark}[theorem]{Remark}

\newtheorem{problem}[theorem]{Problem}

\title[mw-convergence]{multiplicative weak convergence in Banach $f$-algebras}

\date{\today}

\keywords{Banach lattices, $f$-algebra, multiplicative weak convergence, multiplicative order convergence, multiplicative weak convergence, multiplicative weak star convergence.}
\subjclass[2010]{xxx,xxx}

\author[Z. Wang]{Zhangjun Wang}
\address{School of Mathematics, Southwest Jiaotong University,
Chengdu, Sichuan,
China, 610000.}
\email{zhangjunwang@my.swjtu.edu.cn}

\author[Z. Chen]{Zili Chen}
\address{School of Mathematics, Southwest Jiaotong University, Chengdu, Sichuan,
China, 610000.}
\email{zlchen@swjtu.edu.cn}

\author[J. Chen]{Jinxi Chen}
\address{School of Mathematics, Southwest Jiaotong University, Chengdu, Sichuan,
	China, 610000.}
\email{jinxichen@swjtu.edu.cn}

\begin{document}

\begin{abstract}
A net $(x_\alpha)$ in an $f$-algebra $E$ is called multiplicative order convergent to $x\in E $ if
$|x_\alpha -x|u\xrightarrow{o}0$ for all $u\in E_+$.
A net $(x_\alpha)$ in a Banach $f$-algebra $E$ is called multiplicative norm convergent to $x\in E $ if
$|x_\alpha -x|u\rightarrow0$ for all $u\in E_+$.
In this paper, we study this convergence in Banach $f$-algebra and its dual, 
A net $(x_\alpha)$ in a Banach $f$-algebra $E$ is called multiplicative weak convergent to $x\in E $ if
$|x_\alpha -x|u\xrightarrow{w}0$ for all $u\in E_+$.
\end{abstract}

\maketitle

\section{Introduction}
A net $(x_\alpha)$ in a Riesz space $E$ is called order convergent to $x\in E $ if
$|x_\alpha -x|\leq y_\beta$ for $y_\beta\downarrow0\in E_+$. 
A net $(x_\alpha)$ in a Riesz space $E$ is called unbounded order convergent to $x\in E $ if
$|x_\alpha -x|\wedge u\xrightarrow{o}0$ for all $u\in E_+$[4,5].
A net $(x_\alpha)$ in a Banach lattice $E$ is called unbounded norm convergent to $x\in E $ if
$|x_\alpha -x|\wedge u\xrightarrow{\Vert  \Vert}0$ for all $u\in E_+$[6,7].
A net $(x_\alpha)$ in a Banach lattice $E$ is called unbounded absolute weak convergent to $x\in E $ if
$|x_\alpha -x|\wedge u\xrightarrow{w}0$ for all $u\in E_+$[8,9].

In [1],a vector lattice $E$ under an associative multiplication is said to be a $Riesz$ $algebra$
whenever the multiplication makes $E$ an algebra (with the usual properties).A Riesz algebra $E$
is called commutative if $xy = yx$ for all $x,y\in E$.
A Riesz algebra $E$ is called f-algebra if
$E$ has additionally property that $x\wedge y=0$ implies $(xz)\wedge y=(zx)\wedge y=0$ and $xy\in E_+$ for every $x,y\in E_+$.A vector lattice $E$ is called Archimedean whenever $\frac{1}{n}x\downarrow 0$ holds in $E$ for each $x\in E_+$.Every Archimedean $f$-algebra is commutative.In this article, unless otherwise, all vector lattices are assumed to be real and
Archimedean, and so $f$-algebras are commutative.
An $f$-algebra $E$ which is at the same time a
Banach lattice is called a Banach $f$-algebra whenever $\Vert xy\Vert\leq\Vert x\Vert\vert y\Vert$ holds for all
$x,y\in E$.

A net $(x_\alpha)$ in an $f$-algebra $E$ is called multiplicative order convergent to $x\in E $ if
$|x_\alpha -x|u\xrightarrow{o}0$ for all $u\in E_+$ in[11].
A net $(x_\alpha)$ in a Banach $f$-algebra $E$ is called multiplicative norm convergent to $x\in E $ if
$|x_\alpha -x|u\rightarrow0$ for all $u\in E_+$ in[12].

\begin{definition}\label{}
A net $(x_\alpha)$ in a Banach $f$-algebra $E$ is called multiplicative weak convergent to $x\in E $ if
$|x_\alpha -x|u\xrightarrow{w}0$ for all $u\in E_+$ $(x_\alpha\xrightarrow{mw}x)$.

A net $(x^{'}_\alpha)$ in a Banach $f$-algebra $E^{'}$ is called multiplicative weak star convergent to $x^{'}\in E $ if
$|x^{'}_\alpha -x^{'}|u\xrightarrow{w^*}0$ for all $u^{'}\in E^{'}_+$ $(x^{'}_\alpha\xrightarrow{mw^*}x^{'})$.	
\end{definition}

\begin{remark}\label{}
For a net $(x\alpha)$ in a Banach $f$-algebra $E$,$(x_\alpha\xrightarrow{mw}x)$ implies $(x_\alpha y\xrightarrow{mw}xy)$for all
$y\in E$ because of $|x_\alpha y-xy|u=|x_\alpha -x||y|u$ for all $u\in E_+$. The converse
holds true in Banach $f$-algebras with the multiplication unit. Indeed, assume$(x_\alpha y\xrightarrow{mw}xy)$for each $y\in E$. Fix $u\in E_+$, so, $|x_\alpha -x|u=|x_\alpha e-xe|u\xrightarrow{w}0$. Similarily, the $mw^*$-convergence has those properties.
\end{remark}

\begin{remark}\label{}
In Banach $f$-algebras, the weak convergence does not implies the $mw$-convergence, unless it has Schur property. The converse holds true in Banach $f$-algebras with the multiplication unit. Assume$(x_\alpha y\xrightarrow{mw}xy)$for each $y\in E$, then $|x_\alpha -x|=|x_\alpha -x|e\xrightarrow{w}0$, so
$x_\alpha\xrightarrow{w}x$.
\end{remark}

\begin{remark}\label{}
In order continuous Banach $f$-algebras, order convergence and $mo$-convergence imply the $mn$-convergence.
\end{remark}
\section{Results}\label{}

\begin{lemma}\label{}
Let $(x_\alpha)$and $(y_\alpha)$ be two nets in a Banach $f$-algebra $E$. Then the following holds true:

$(1)$ $x_\alpha\xrightarrow{mw}x$ iff $x_\alpha -x\xrightarrow{mw}0$ iff $|x_\alpha -x|\xrightarrow{mw}0$;

$(2)$ if $x_\alpha\xrightarrow{mw}x$ then $y_\beta\xrightarrow{mw}x$ for each subnet $(y_\beta)$ of $(x_\alpha)$;

$(3)$ suppose $x_\alpha\xrightarrow{mw}x$ and $y_\beta\xrightarrow{mw}y$, then $ax_\alpha +by_\beta\xrightarrow{mw}ax+by$ for any $a,b\in R$;

$(4)$ if $x_\alpha\xrightarrow{mw}x$ and $x_\alpha\xrightarrow{mw}y$, then $x=y$;

$(5)$ if $x_\alpha\xrightarrow{mw}x$ then $|x_\alpha|\xrightarrow{mw}|x|$.

The $mw^*$-convergence also has those properties.
\end{lemma}

\begin{proof}
We only need $|x-y|\leq|x-x_\alpha|+|x_\alpha-y|$ and $\big||x_\alpha|-|x|\big|\leq|x_\alpha-x|$.
\end{proof}

\begin{lemma}\label{}
For a Banach $f$-algebra $E$ is order continuous, the $mo$-convergence implies $mw$-convergence.
\end{lemma}
\begin{proof}
According to [12,remark 1.2], we have the conclusion.
\end{proof}

\begin{lemma}\label{}
Every disjoint decreasing sequence in a Banach $f$-algebra $mw$-converges to zero.
\end{lemma}

\begin{proof}
Suppose $(x_n)$ is a disjoint and decreasing sequence in an Banach $f$-algebra $E$. So, $|x_n|u$
is also a disjoint sequence in $E$ for all $u\in E_+$ by[2,definition 2.53 and 3,definition 3.1.8].Fix $u\in E_+$, by[8.lemma 2], we have $|x_n|u\xrightarrow {uaw}0$ in $E$. So, $|xn|u\wedge w\xrightarrow{w}0$ in $E$ for all $w\in E_+$. Thus, in particular for fixed $n_0$, taking $w$ as $|x_{n_{0}}|u$. Then, for all $n\geq n_0$, we get
$$|x_n|u=|x_n|u\wedge|x_{n_{0}}|u=|x_n|u\wedge w\xrightarrow{w}0$$ Since $|x_n|u\leq|x_{n_{0}}|u$, therefore, $x_n\xrightarrow{mw}0$ in $E$.
\end{proof}

\begin{lemma}\label{}
Let $E$ be an Banach $f$-algebra, $B$ be a projection band of $E$ and $P_B$ be the corresponding band projection. Then $x_\alpha\xrightarrow{mw}x$ in $E$ implies $P_B(x_\alpha)\xrightarrow{mw}P_B(x)$ in both $E$ and $B$.
\end{lemma}

\begin{proof}
The proof is similarly to [11,proposition 2.7].
\end{proof}

\begin{lemma}\label{}
Let $(x_\alpha)$ be a net in a Banach $f$-algebra $E$ with order continuous norm. Then we have that 

$(1)$ $0\leq x_\alpha\xrightarrow{mw}x$ implies $x\in E_+$;

$(2)$ if $(x_\alpha)$ is monotone and $x_\alpha\xrightarrow{mw}x$ then $x_\alpha\xrightarrow{w}x$.
\end{lemma}

\begin{proof}
$(1)$: Assume $(x_\alpha)$ consists of non-zero elements and $mw$-converges to $x\in E$. Then, by Lemma 2.1(5), we have $x_\alpha=x^+_\alpha\xrightarrow{mw}x^+\geq0$. Therefore, we have $x\in E_+$.

$(2)$: For a fixed $\alpha$ ,we have $x_\beta-x_\alpha\in E_+$ for $\beta\geq\alpha$. By (1), we have $x_\beta-x_\alpha\xrightarrow{mw}x-x_\alpha\in E_+$, so $x_\alpha\uparrow x$, since $E$ is order continuous, therefore, $x_\alpha\xrightarrow{w}x$.
\end{proof}

The lattice operations in Banach lattice $f$-algebras are $mw$-continuous in the following sense.

\begin{proposition}\label{}
Let $(x_\alpha)_{\alpha\in A}$ and $(y_\beta)_{\beta\in B}$ be two nets in a Banach $f$-algebra $E$. If $x_\alpha\xrightarrow{mw}x$ and $y_\beta\xrightarrow{mw}y$ then $(x_\alpha\vee y_\beta)_{(\alpha,\beta)\in A\times B}\xrightarrow{mw}x\vee y$.($\wedge$ and $||$ are similarily)
\end{proposition}
\begin{proof}
Assume $x_\alpha\xrightarrow{mw}x$ and $y_\beta\xrightarrow{mw}y$. Since we have
$$|x_\alpha\vee y_\beta-x\vee y|u\leq |x_\alpha\vee y_\beta-x_\alpha\vee y|u+|x_\alpha\vee y-x\vee y|u$$
$$\leq|y_\beta-y|u+|x_\alpha-x|u\xrightarrow{w}0$$
for every $u\in E_+$.That is, $(x_\alpha\vee y_\beta)_{(\alpha,\beta)\in A\times B}\xrightarrow{mw}x\vee y$.
\end{proof}

The multiplication in $f$-algebra is $mw$-continuous in the following sense.
\begin{proposition}\label{}
Let $E$ be a Banach lattice $f$-algebra, and $(x_\alpha)_{\alpha\in A}$ and $(y_\beta)_{\beta\in A}$ be two nets in $E$. If $x_\alpha\xrightarrow{mw}x$ and $y_\beta\xrightarrow{mw}y$ for some $x,y\in E$ ,$(x_\alpha)$ or $(y_\beta)$ is monotone, then we have $x_\alpha y_\beta\xrightarrow{mw}xy$.
\end{proposition}

\begin{proof}
Assume $x_\alpha\xrightarrow{mw}x$ and $y_\beta\xrightarrow{mw}y$, then we have $|x_\alpha-x|u\xrightarrow{w}0$ and $|y_\beta-x|u\xrightarrow{w}0$ for every $u\in E_+$.Since
$$|x_\alpha y_\beta-xy|u=|x_\alpha y_\beta-x_\alpha y+x_\alpha y-xy|u$$
$$\leq||x_\alpha-x+x||y_\beta-y|u+|x_\alpha-x||y|u$$
$$\leq|x_\alpha-x||y_\beta-y|u+|y_\beta-y||x|u+|x_\alpha-x||y|u$$
The second and the third terms are weak converges to zero, we show first term is weak converges to zero.
Assume $(x_\alpha)$ is increasing, then $|x_\alpha-x|\leq 2|x|$ and $2|y_\beta-y||x|u$ is weak converges to zero, so we have the conclusion.
\end{proof}

\begin{theorem}\label{}
Let $E$ be a order continuous Banach $f$-algebra with a multiplicative unit $e$ and $(x_n)\downarrow$ be a sequence in $E$. Then $x_n\xrightarrow{mw}0$ iff $|x_n|(u\wedge n)\xrightarrow{w}0$ for all $u\in E_+$.
\end{theorem}

\begin{proof}
For the forward implication, assume $x_n\xrightarrow{mw}0$, then $|x_n|u\xrightarrow{w}0$ for all $u\in E_+$. Since $|x_n|(u\wedge e)\leq|x_n|u\xrightarrow{w}0$,therefore, $|x_n|(u\wedge n)\xrightarrow{w}0$.

For the reverse implication, by applying [2,theorem 2.57],we have
$$|x_n|u\leq |x_n|(u-u\wedge ne)+|x_n|(u\wedge ne)\leq\frac{1}{n}u^2|x_n|+n|x_n|(u\wedge e)$$
Since $(x_n)\downarrow$ and $E$ is a order continuous Banach $f$-algebra, we have the first term converges weakly to zero, and it is similarily to the proof of [6,lemma 2.11], the second term weak convergent to zero, so $x_n\xrightarrow{mw}0$.
\end{proof}

It is similarily to [12,proposition 2.4], we have the following result.

\begin{theorem}\label{}
Let $0\leq(x_\alpha)_{\alpha\in A}\downarrow$ be a net in a Banach $f$-algebra $E$ with a quasi-interior point $e$. Then $x_\alpha\xrightarrow{mw}0$ iff $x_\alpha e\xrightarrow{w}0$.
\end{theorem}

\begin{corollary}\label{}
Let $0\leq(x_\alpha)_{\alpha\in A}\downarrow$ be a net in an order continuous Banach $f$-algebra $E$ with a weak order unit $e$. Then $x_\alpha\xrightarrow{mw}0$ iff $x_\alpha e\xrightarrow{w}0$.
\end{corollary}

\begin{corollary}\label{}
Let $0\leq(x_\alpha)_{\alpha\in A}\downarrow$ be a net in a separable Banach $f$-algebra $E$. Then $x_\alpha\xrightarrow{mw}0$ iff $x_\alpha e\xrightarrow{w}0$.
\end{corollary}

\begin{corollary}\label{}
Let $0\leq(x_\alpha)_{\alpha\in A}\downarrow$ be a net in a Banach $f$-algebra $E$ with a quasi-interior point $e$. Then $x_\alpha\xrightarrow{mw}0$ iff $x_\alpha (e+u)\xrightarrow{w}0$ for all $u\in E_+$.
\end{corollary}

\begin{definition}\label{}
A subset $A$ of $E$ is called a $f$-weak-almost order bounded if for any $\epsilon>0$, there is $u\in E_+,f\in E_+^{'}$ such that $f(|x|-u|x|)\leq\epsilon$. 
\end{definition}
Next, we have the following work, it is similarily to [10,proposition 2.8].

\begin{theorem}\label{}
Let $E$ be a Banach $f$-algebra. If $(x_\alpha)$ is $f$-weak-almost order bounded and $mw$-converges to $x$, then $x_\alpha\xrightarrow{|\sigma|(E,E^{'})}x$.
\end{theorem}

\begin{proof}
If $(x_\alpha)$ is $f$-weak-almost order bounded net. Then the net $(|x_\alpha-x|)$ is also. For any $\epsilon>0$, there exists $u>0,f\in E^{'}_+$ such that 
$$f(|x_\alpha-x|-u|x_\alpha-x|)\leq\epsilon.$$
Since $x_\alpha\xrightarrow{mw}x$, we have $|x_\alpha-x|u\xrightarrow{w}0$. Therefore, we have $x_\alpha\xrightarrow{|\sigma|(E,E^{'})}x$.
\end{proof}

\begin{definition}\label{}
Let $(x_\alpha)$ be a net in a Banach lattice $f$-algebra $E$. Then 

$(1)$  $(x_\alpha)$ is said to be $mw$-Cauchy if the net
$(x_\alpha-x_{\alpha^{'}})_{(\alpha,\alpha^{'})\in A\times A}$ $mw$-converges to 0.

$(2)$ $E$ is called $(\sigma)$-$mw$-complete if every $mw$-Cauchy net(sequence) in $E$ is $mw$-convergent.

$(3)$ $E$ is called $mw$-continuous if $x_\alpha\xrightarrow{o}0$ implies that $x_\alpha\xrightarrow{mw}0$.	
\end{definition}

\begin{lemma}\label{}
A Banach $f$-algebra is $mw$-continuous iff $x_\alpha\downarrow0$ implies $x_\alpha\xrightarrow{mw}0$.
\end{lemma}

\begin{proof}
Let $(x_\alpha)\xrightarrow{o}0$, we have there exists $y_\beta\downarrow0$ and $|x_\alpha|\leq y_\beta$. Since $y_\beta\downarrow0$, so $y_\beta\xrightarrow{mw}0$, we have $x_\alpha\xrightarrow{mw}0$.	
\end{proof}

\begin{theorem}\label{}
Let $E$ be an $mw$-complete Banach $f$-algebra. Then the following statements are quivalent:

$(1)$ $E$ is $mw$-continuous;

$(2)$ if $0\leq x_\alpha\uparrow\leq x$ holds in $E$ then $(x_\alpha)$ is an $mw$-Cauchy net;

$(3)$ $x_\alpha\downarrow0$ implies $x_\alpha\xrightarrow{mw}0$ in $E$.
\end{theorem}

\begin{proof}
$(1)\Rightarrow(2)$: By the proof of [2,lemma 1.37], we have $(y_\beta-x_\alpha)_{\alpha,\beta}\downarrow0$; according to lemma 2.16, we have $(y_\beta-x_\alpha)_{\alpha,\beta}\xrightarrow{mw}0$, so $(x_\alpha)$ is an $mw$-Cauchy net.

$(2)\Rightarrow(3)$: Fix arbitrary $\alpha_0$, we have $0\leq(x_{\alpha_0}-x_\alpha)_{(\alpha\geq\alpha_0)}\uparrow\leq x_{\alpha_0}$. So $(x_{\alpha_0}-x_\alpha)_{(\alpha\geq\alpha_0)}$ is $mw$-Cauchy net ,so $(x_{\alpha^{'}}-x_\alpha)\xrightarrow{mw}0$. Since $E$ is $mw$-complete, so $x_\alpha\xrightarrow{mw}x$. Because of lemma 2.5, we have $x=0$, therefore, $x_\alpha\xrightarrow{mw}0$.

$(3)\Rightarrow(1)$: By lemma 2.16.
\end{proof}

\begin{corollary}\label{}
$(\sigma)$-$mw$-complete also has those properties of the last theorem.	
\end{corollary}

\begin{corollary}\label{}
Every $mw$-continuous and $mw$-complete Banach $f$-algebra is Dedekind complete.	
\end{corollary}

\begin{proof}
Suppose $E$ is $mw$-continuous and $mw$-complete and $0\leq x_\alpha\uparrow\leq y$ in $E$. By theorem 2.17, $(x_\alpha)$ is $mw$-Cauchy, so $x_\alpha\xrightarrow{mw}x$ and by the proof of lemma 2.5, we have $x_\alpha\uparrow x$, so $E$ is Dedekind complete.
\end{proof}

It was observed in [7,8,10], we now turn our attention to a topology on Banach $f$-algebras. The sets of the form
$$V_{u,\epsilon,f}= \{x\in E:f(|x|u)<\epsilon \},$$
where $u\in E_+,\epsilon>0,f\in E_+^{'}$ form a base of zero neighborhoods for a Hausdorff topology, and the convergence in this topology is exactly the $mw$-convergen.

Similarily to [7,8,10], we these conclusions.

\begin{lemma}\label{}
$V_{u,\epsilon,f}$ is either contained in $[-u,u]$ or contains a non-trivial ideal.	
\end{lemma}

\begin{lemma}\label{}
If $V_{u,\epsilon,f}$ is contained in $[-u,u]$, then $u$ is a strong unit.	
\end{lemma}
Similarily to [10]
\begin{theorem}\label{}
Let $E$ be a Banach lattice.$E^{'}$ has a strong order unit when one of the following conditions is valid:

(1) $mw$-topology agrees with norm topology.

(2) $mw$-topology agrees with weak topology.	
\end{theorem}

\begin{problem}\label{}
How to describe the compactness of $mw$-topology.

When the $mw$-topology is metrizability.
\end{problem}	
\noindent \textbf{Acknowledgement.} xxxxxx

\end{document}